\RequirePackage{fix-cm}

\documentclass{amsart}

\usepackage{mathptmx}

\usepackage[all]{xy}
\usepackage{url}
\usepackage{amssymb}
\usepackage{amsmath}
\usepackage{array}

\theoremstyle{plain}
\newtheorem{theorem}{Theorem}[section]
\newtheorem{proposition}[theorem]{Proposition}
\newtheorem{lemma}[theorem]{Lemma}
\newtheorem{corollary}[theorem]{Corollary}
\theoremstyle{definition}
\newtheorem{definition}[theorem]{Definition}
\newtheorem{example}[theorem]{Example}
\theoremstyle{remark}
\newtheorem{remark}[theorem]{Remark}

\newcommand{\propref}[1]{Proposition~\ref{#1}}
\newcommand{\thmref}[1]{Theorem~\ref{#1}}
\newcommand{\lemmaref}[1]{Lemma~\ref{#1}}
\newcommand{\cororef}[1]{Corollary~\ref{#1}}
\newcommand{\exref}[1]{Example~\ref{#1}}
\newcommand{\secref}[1]{Section~\ref{#1}}
\newcommand{\remref}[1]{Remark~\ref{#1}}

\newcommand{\PP}{\mathbb{P}}
\newcommand{\CC}{\mathbb{C}}
\newcommand{\ZZ}{\mathbb{Z}}
\newcommand{\aut}[1]{\operatorname{Aut} (#1)}
\newcommand{\GL}[2]{\operatorname{GL}_{#1} (#2)}
\newcommand{\PGL}[2]{\operatorname{PGL}_{#1} (#2)}
\newcommand{\Sp}[2]{\operatorname{Sp}_{#1} (#2)}
\newcommand{\SL}[2]{\operatorname{SL}_{#1} (#2)}
\newcommand{\orb}[2]{\operatorname{orb}_{#1} (#2)}
\newcommand{\rank}[1]{\operatorname{rank} (#1)}
\newcommand{\pic}[1]{\operatorname{Pic} (#1)}
\newcommand{\piczero}[1]{\operatorname{Pic}^0 (#1)}
\newcommand{\discr}[1]{\operatorname{discr} (#1)}
\newcommand{\NL}[1]{\operatorname{NL}_{#1}}
\newcommand{\NS}[1]{\operatorname{NS}(#1)}

\renewcommand{\dim}[1]{\operatorname{dim} (#1)}

\newcolumntype{K}{>{\centering\arraybackslash$}p{0.4cm}<{$}}

\begin{document}

\title{Curves on Heisenberg invariant quartic surfaces in projective 3-space}

\author{David Eklund}
\address{Department of Mathematics, KTH, 100 44 Stockholm, Sweden}
\email{daek@math.kth.se}
\urladdr{https://people.kth.se/~daek/}

\begin{abstract}
This paper is about the family of smooth quartic surfaces $X \subset
\PP^{3}$ that are invariant under the Heisenberg group $H_{2,2}$. For
a very generic $X$, we show that the Picard number of $X$ is 16 and
determine its Picard lattice. It turns out that a very generic $X$
contains 320 irreducible conics which generate the Picard group of
$X$.
\end{abstract}

\maketitle

\section{Introduction}
Let $A$ be an Abelian surface over $\CC$, that is a projective group
variety of dimension 2. The subgroup $A_2=\{a \in A:2a=0\}$ has order
16 and therefore $A_2 \cong (\ZZ / 2\ZZ)^4$. The
involution $i:A \rightarrow A:a \mapsto -a$ induces a $\ZZ /
2\ZZ$ action on $A$ and the quotient, which we denote by $K_A$,
is called the Kummer surface of $A$. If $A$ admits a certain kind of
line bundle (see \secref{sec:Kummer}) there is an induced map $A
\rightarrow \PP^3$ which factors through an embedding $K_A \rightarrow
\PP^3$ such that the image of the Kummer surface is a quartic with 16
nodes. Moreover, the natural action of $A_2$ on $K_A$ extends to a
linear action on $\PP^3$. In one set of coordinates $(x,y,z,w)$ on
$\PP^3$ this action is given by identifying $A_2$ with the subgroup of
$\aut{\PP^{3}}$ which is generated by the following four
transformations:
\begin{center}
\begin{tabular}{ll}
$(x,y,z,w) \mapsto (z,w,x,y)$, & $(x,y,z,w) \mapsto (y,x,w,z)$, \\
$(x,y,z,w) \mapsto (x,y,-z,-w)$, & $(x,y,z,w) \mapsto (x,-y,z,-w)$.
\end{tabular}
\end{center}
The subject of this paper is the family of all quartic surfaces in
$\PP^{3}$ which are invariant under these transformations. We will
refer to such surfaces as \emph{Heisenberg invariant quartics} or just
\emph{invariant quartics}. This family of quartics appeared in the
classical treatises \cite{H,J} and also in several later works
\cite{BN,Mum,V}. The family is parameterized by $\PP^{4}$ and the
subfamily of Kummer surfaces described above constitutes a Zariski
open dense subset of a hypersurface $S_3 \subset \PP^{4}$ known as the
Segre cubic. However, the general Heisenberg invariant quartic is
smooth. Special smooth members of the family appear throughout the
literature, for example the Fermat quartic $x^4+y^4+z^4+w^4=0$ which
is studied in \cite{Se} and the pencil $x^4+y^4+z^4+w^4+\lambda
xyzw=0$, which has been studied from the point of view of mirror
symmetry \cite{Do} and the periodicity of the Clebsch contravariant
\cite{E}. In \cite{BN}, Barth and Nieto study the locus of Heisenberg
invariant quartics that contain a line and find a quintic threefold
$N_5 \subset \PP^{4}$ such that the general point corresponds to a
desingularized Kummer surface of an Abelian surface with a $(1,3)$
polarization. Prior to that, these surfaces had been discovered by
Traynard \cite{T} and discussed by Godeaux \cite{G} and Naruki
\cite{Na}. The present paper was motivated by the following question:
which Heisenberg invariant quartics contain a conic? It turns out that
a general invariant quartic contains 320 smooth conics. The conics are
found by a direct computation using the geometry of the
family. Another result of this paper is that a very general member of
the family of surfaces has Picard number 16. This is in accordance
with the fact that certain moduli spaces of K3 surfaces whose Picard
lattices contain (an isomorphic copy of) a fixed lattice $M$ have
dimension $20-\rank{M}$. The group action on the surfaces induces an
action on the Picard group and we also show that, for a very general
surface in the family, the sublattice of invariant divisor classes is
generated by the class of the hyperplane section. In particular, any
invariant curve on such a surface is a complete intersection. Next we
determine the Picard lattice of a very general Heisenberg invariant
quartic and show that it is generated by the 320 smooth conics. This
is done by computing the configuration of the 320 conics as well as
using some general facts on the existence of curves on Kummer
surfaces.

The paper is organized as follows. \secref{sec:family} and
\secref{sec:moduli} set the notation and review some results that we
need for the sequel. In \secref{sec:picard} we determine the Picard
number of a very generic invariant quartic. \secref{sec:conics}
concerns the existence of the 320 conics. In \secref{sec:lines} and
\secref{sec:NS} we study configurations of lines on Heisenberg
invariant quartic surfaces and use the results to determine the
configuration of the 320 smooth conics on a very general surface. In
\secref{sec:picardgroup} we determine the Picard lattice of a very
general surface in the family up to isomorphism and show that it is
generated by the 320 smooth conics.

\section*{Acknowledgements}
I thank Kristian Ranestad for guidance and inspiration, as well as
making it possible for me spend two months during the spring of 2009
at the University of Oslo. Thanks to Igor Dolgachev for pointing out
relevant references. During the course of this work, the software
packages \emph{Bertini} \cite{BHSW} and \emph{Macaulay 2} \cite{GS}
were used for experimentation. This is a pre-print of an article
published in European Journal of Mathematics, the final version is
available at \url{https://doi.org/10.1007/s40879-018-0216-2}.

\section{The family of invariant quartics} \label{sec:family}

\subsection{Kummer surfaces} \label{sec:Kummer}

We begin with an overview of Kummer surfaces, for proofs and notation
see \cite{BL,GD2,H,Mum}. Let $A$ be an Abelian surface over
$\CC$. The subgroup of 2-torsion points $A_2=\{a \in A:2a=0\}$ has
order 16 and hence $A_2 \cong (\ZZ / 2\ZZ)^4$. To an
element $a \in A$ we associate a translation map $t_a:A \rightarrow
A:x \mapsto x+a$. The involution $i:A \rightarrow A:a \mapsto -a$
induces a $\ZZ / 2\ZZ$ action on $A$ and the quotient
$K_A=A/\{1,i\}$ is an algebraic surface called the Kummer surface of
$A$. Let $\pi:A\rightarrow K_A$ be the projection. The Kummer surface
has 16 singular points, namely $\pi(A_2)$, and the action of $A_2$ on
$A$ by translations induces an action on $K_A$. For an ample line
bundle $L$ on $A$ we define the Heisenberg group $$H(L)=\{a \in
A:t_a^{\ast} L \cong L\} \subset A,$$ and the set
$$\mathcal{G}(L)=\{(x,\phi): x \in H(L), \;
\phi:L\overset{\cong}{\rightarrow} t_x^{\ast} L\}.$$ If $(x,\phi) \in
\mathcal{G}(L)$ and $(y,\psi) \in \mathcal{G}(L)$ then there is an
induced isomorphism $t_x^* \psi :t_x^* L \rightarrow
t_x^*(t_y^*L)$. Using that $t_x^*(t_y^*L)=t_{x+y}^* L$ we put a group
structure on $\mathcal{G}(L)$ by letting $$(y,\psi)(x,\phi)=(x+y,t_x^*
\psi \circ \phi).$$ These two groups are connected by an exact
sequence,
$$1\rightarrow \CC^{\ast} \rightarrow \mathcal{G}(L) \rightarrow H(L)
\rightarrow 1,$$ where $\mathcal{G}(L) \rightarrow
H(L):(x,\phi)\mapsto x$. The kernel of the map $\mathcal{G}(L)
\rightarrow H(L)$ is the group of automorphisms of $L$, which is the
multiplicative group $\CC^{\ast}$ acting by multiplication by
constants. We will consider the case where $A$ has a principal
polarization $L'$, that is $L'$ is an ample line bundle on $A$ whose
elementary divisors are both equal to 1. In addition we assume that
$L'$ is symmetric and irreducible. The former means that $i^*L' \cong
L'$ where $i:A \rightarrow A:a \mapsto -a$, and the latter means that
the polarized Abelian surface $(A,L')$ does not split as a product of
elliptic curves. We say that the line bundle $L=L' \otimes L'$ is of
type $(2,2)$. Then $H(L)=A_2$, $\dim{\Gamma(A,L)}=4$ and if $D$ is a
divisor on $A$ that corresponds to $L$, then $D^2=8$. The group
$\mathcal{G}(L)$ has an action on the space of global sections of $L$:
for $z=(x,\phi)$ and $s \in \Gamma(A,L)$, $\phi$ induces a section
$\phi(s)$ of $t_x^*L$ and the translation $t_{-x}^*$ induces a section
$t_{-x}^*(\phi(s))$ of $L=t_{-x}^*(t_x^*L)$. Thus we put
$zs=t_{-x}^{\ast}(\phi(s))$. This makes $\Gamma(A,L)$ into an
irreducible $\mathcal{G}(L)$ module such that $\CC^{\ast}$ acts by
rescaling. After choosing a basis of $\Gamma(A,L)$, this gives a
faithful linear action of $H(L)=A_2$ on $\PP^{3}$. The rational map
$A\rightarrow \PP^{3}$ induced by $L$ is defined everywhere and
factors through an embedding $K_A \rightarrow \PP^{3}$ and the image
of the Kummer surface is a quartic with 16 nodes. This is the maximal
number of nodes of a quartic surface in $\PP^{3}$ and any quartic in
$\PP^{3}$ with 16 nodes is a Kummer surface. For each of the nodes
$p$, there is a plane $P$ which contains $p$ and 5 other points in the
orbit of $p$ under $H(L)$. Moreover, these 6 points lie on a
non-degenerate conic $C$ and the plane $P$ touches the Kummer surface
along $C$. We say that $P$ is a trope. In total we have 16 tropes
arising this way and together with the 16 nodes they form Kummer's
$16_6$ configuration: each plane contains 6 points and there are 6
planes through each point. Now, the $H(L)$ action on $\PP^{3}$
restricts to the given $A_2$ action on $K_A$ and the embedded Kummer
surface is thus invariant under the linear action on the ambient
space. The Stone-von-Neumann-Mumford theorem states that, up to
isomorphism, $\mathcal{G}(L)$ has a unique irreducible representation
such that $\CC^{\ast}$ acts by rescaling. Throughout the paper we will
fix coordinates on $\PP^3$ and the particular action of $H(L)$ given
in the introduction. Fixing the action of $H(L)$ merely reflects a
choice of coordinates though, see for example \cite{GD2} Lemma 1.52.

\subsection{The Heisenberg group} \label{sec:Hgroup}

Let $(x,y,z,w)$ be coordinates on $\CC^{4}$ and consider the following
four elements of $\SL{4}{\CC}$:
\begin{center}
\begin{tabular}{l}
$\sigma_1:(x,y,z,w) \mapsto (z,w,x,y)$ \\ 
$\sigma_2:(x,y,z,w) \mapsto (y,x,w,z)$ \\ 
$\tau_1:(x,y,z,w) \mapsto (x,y,-z,-w)$ \\ 
$\tau_2: (x,y,z,w) \mapsto (x,-y,z,-w).$
\end{tabular}
\end{center}
Let $H_{2,2}$ be the subgroup of $\SL{4}{\CC}$ generated by $\sigma_1,
\sigma_2, \tau_1$, and $\tau_2$. This group is of order 32 and the
center and commutator of $H_{2,2}$ are both equal to $\{1,-1\}$, where
$1$ denotes the identity matrix of size $4$. Let
$$H = H_{2,2}/\{1,-1\} \subset \aut{\PP^{3}},$$ a group of order
16. Since every element of $H$ has order 2, it follows that
$$H \cong (\ZZ / 2\ZZ)^4.$$ In the sequel we will refer to $H$ as the \emph{Heisenberg group}. 

We consider $H$ as an $\mathbb{F}_2$-vector space $(\mathbb{F}_2)^4
\cong H$. It carries a symplectic form given by $\langle g,h
\rangle=0$ if $g,h \in H_{2,2}$ commute and $\langle g,h \rangle=1$ if
they anticommute.  For each element $g \in H_{2,2}$, $g \neq \pm 1$,
the fixed points of $g \in H$ form two skew lines in $\PP^{3}$. These
correspond to the eigenspaces of $g$ in $\CC^{4}$, with eigenvalues
$\pm 1$ or $\pm \sqrt{-1}$. In total we have 30 such lines in
$\PP^{3}$ which will be called the \emph{fix lines}.  If $h \in
H_{2,2}$ commutes with $g$, then $h$ leaves the fix lines belonging to
$g$ invariant and, if $h$ anticommutes with $g$, then it flips the fix
lines of $g$. Two fix lines belonging to $g,h \in H_{2,2}$ where $g
\neq \pm h$, intersect if and only if $g$ and $h$ commute.

Let $\mathfrak{S}_n$ denote the symmetric group on $n$ letters and let
$N$ be the normalizer of $H_{2,2}$ in $\SL{4}{\CC}$, $$N=\{n \in
\SL{4}{\CC}:nH_{2,2}n^{-1}=H_{2,2}\}.$$ Let $n \in N$. The reason for
introducing $N$ is that if a subset $X \subseteq \PP^{3}$ is invariant
under $H$ and $n \in N$, then $nX$ is also invariant. In fact, for any
$g \in H_{2,2}$ there is an $h \in H_{2,2}$ such that $gn=nh$, and
hence $g(nX)=n(hX)=nX$. We now proceed to give a useful relation
between $N$ and $\mathfrak{S}_6$. Since $N$ acts on $H_{2,2}$ by
conjugation, and acts as the identity on the center $\{1,-1\}$, $N$
acts on the vector space $H_{2,2} / \{1,-1\} \cong (\mathbb{F}_2)^4
$. This action is linear and the transformations preserve the
symplectic form. Thus we have a map $$\phi:N \rightarrow
\Sp{4}{\mathbb{F}_2}.$$ The kernel obviously contains $H_{2,2}$. In
fact it contains the group $\langle \pm i, H_{2,2} \rangle$ of order
64 which is generated by $H_{2,2}$ and multiplication by $\pm i$,
where $i=\sqrt{-1}$. By \cite{Di} Theorem 118, $\Sp{4}{\mathbb{F}_2}
\cong \mathfrak{S}_6$ and we get a sequence
$$1 \rightarrow \langle \pm i, H_{2,2} \rangle \rightarrow N
\rightarrow \mathfrak{S}_6 \rightarrow 1.$$ In \cite{N} it is shown
that this sequence is exact, and thus $N / \langle \pm i, H_{2,2}
\rangle \cong \mathfrak{S}_6$.

\subsection{The family of surfaces}
Consider the following elements of $\mathbb{C}[x,y,z,w]$: 
$$g_0 = x^4+y^4+z^4+w^4, \quad g_1 = 2(x^2y^2+z^2w^2), \quad g_2 = 2(x^2z^2+y^2w^2),$$ $$g_3 = 2(x^2w^2+y^2z^2), \quad  g_4 = 4xyzw.$$
For $\lambda=(A,B,C,D,E) \in \CC^{5}$ define
\begin{eqnarray} \label{eq:inv_pol}
F_{\lambda} = Ag_0+Bg_1+Cg_2+Dg_3+Eg_4.
\end{eqnarray}
Then $\{F_{\lambda}:\lambda \in \CC^{5}, \lambda \neq 0\} \subset \mathbb{C}[x,y,z,w]$
is the set of all homogeneous quartic polynomials in $\{x,y,z,w\}$
invariant under $H_{2,2}$. Now let $\mathcal{X} \rightarrow \PP^{4}$
be the corresponding family of Heisenberg invariant quartic surfaces:
$$\PP^{4} \times \PP^{3} \supset \mathcal{X} =
\{(\lambda,p):F_{\lambda}(p)=0\}.$$ There is a linear embedding of
$\PP^{4}$ into $\PP^{5}$ which better exposes the symmetry of the
situation. The normalizer $N$ of $H_{2,2}$ in $\SL{4}{\CC}$ acts on
the 5-dimensional vector space of $H_{2,2}$-invariant polynomials via
the natural action on $\CC^4$. Following \cite{BN,H,N} we use
six invariant quartic polynomials $t_0,\dots,t_5$, which satisfy the
relation $\sum_i t_i =0$, to embed the parameter space in $\PP^{5}$:
\begin{align*}
t_0 &= \frac{1}{3}g_0-g_1-g_2-g_3, & t_1 &= \frac{1}{3}g_0 -g_1+g_2+g_3, \\
t_2 &= \frac{1}{3}g_0+g_1-g_2+g_3, & t_3 &= \frac{1}{3}g_0+g_1+g_2-g_3, \\
t_4 &= -\frac{2}{3}g_0+2g_4, & t_5 &= -\frac{2}{3}g_0-2g_4.
\end{align*}
We use homogeneous coordinates $(u_0,u_1,u_2,u_3,u_4,u_5)$ on
$\PP^{5}$, expressing an invariant quartic polynomial as $\sum_i t_i
u_i$. The polynomials $t_0,t_1,t_2,t_3,t_4,t_5$ generate the space of
all $H_{2,2}$ invariant quartic polynomials. The action of $N/\langle
\pm i,H_{2,2} \rangle \cong \mathfrak{S}_6$ permutes these six
elements up to a sign change and the group acts on $\PP^5$ as the full
permutation group. Let $U \cong \PP^{4}$ be the linear subspace of
$\PP^{5}$ given by
$$U = \{u_0+u_1+u_2+u_3+u_4+u_5=0\} \subset \PP^{5}.$$ The parameter
space $\PP^{4}$ of $\mathcal{X}$ with coordinates $(A,B,C,D,E)$ is
identified with $U$ via the relations
\begin{align*}
A &=-u_4-u_5, \\
B &=-u_0-u_1+u_2+u_3, \\
C &=-u_0+u_1-u_2+u_3, \\
D &=-u_0+u_1+u_2-u_3, \\
E &=2u_4-2u_5.
\end{align*}
The transformation from $U$ to the parameters $(A,B,C,D,E)$ may be written as
\begin{align*}
u_0 &=A-B-C-D, \\
u_1 &=A-B+C+D, \\
u_2 &=A+B-C+D, \\
u_3 &=A+B+C-D, \\
u_4 &=-2A+E, \\
u_5 &=-2A-E.
\end{align*}
The point is that $\mathfrak{S}_6$ acts on $U$ by permuting the coordinates, and on the level of the quartic surfaces the action is by means of projective transformations on $\PP^{3}$. 

Let $\pi:\mathcal{X} \rightarrow U$ be the projection. For a point $u
\in U$ we have a corresponding surface $\mathcal{X}_u=\pi^{-1}(u)$. By
\emph{Heisenberg invariant quartic surface} (or simply \emph{invariant
  quartic}), we shall understand a fiber $\mathcal{X}_u$ for some $u
\in U$. In addition to these there are 15 pencils of quartic surfaces
in $\PP^{3}$ which are invariant under $H$, but whose defining
polynomials are not invariant under $H_{2,2}$ (these will, however, not
be considered).

\subsection{The Segre cubic, the Igusa quartic and the Nieto quintic}
In this section we introduce three hypersurfaces in $\PP^{4}$ relevant for the study of Heisenberg invariant quartics. If a group $G$ acts on a set $A$ we use $\orb{G}{a}$ to denote the orbit of a point $a \in A$. 

First consider the following sets of points, planes and 3-planes in the parameter space $U$:
\begin{align*}
q_0 &=(1,1,1,-1,-1,-1) \\ 
t_0 &= (1,-1,0,0,0,0) \\
p_0 &=\{u_0+u_1=u_2+u_3=u_4+u_5=0\} \\ 
e_0 &=\{u_0+u_1=u_0+u_1+u_2+u_3+u_4+u_5=0\} \\
\mathcal{Q} &= \orb{\mathfrak{S}_6}{q_0} \quad \textrm{(10 points)} \\
\mathcal{T} &= \orb{\mathfrak{S}_6}{t_0} \quad \textrm{(15 points)} \\
\mathcal{P} &= \orb{\mathfrak{S}_6}{p_0} \quad \textrm{(15 planes)} \\
\mathcal{E} &= \orb{\mathfrak{S}_6}{e_0} \quad \textrm{(15 3-planes)}.
\end{align*}
These points, planes and 3-planes relate to a
$\mathfrak{S}_6$-invariant cubic hypersurface $S_3 \subset U \cong
\PP^{4}$ known as the \emph{Segre cubic}. In $\PP^{5}$ it is defined
by the equations $$S_3 =
\{u_0^3+u_1^3+u_2^3+u_3^3+u_4^3+u_5^3=u_0+u_1+u_2+u_3+u_4+u_5=0\}.$$
The Segre cubic has ten nodes, namely the $ \mathfrak{S}_6$-orbit
$\mathcal{Q}$. It also contains the 15 points of $\mathcal{T}$, and
the set of 3-planes $\mathcal{E}$ is the set of tangent spaces to
$S_3$ at these points. For $t \in \mathcal{T}$ the intersection
between $S_3$ and the tangent space $T_t S_3$ is a union of three
2-planes. In total we have 15 such planes in $S_3$, called the
\emph{Segre planes}. This is the set $\mathcal{P}$. In \cite{N2},
Nieto follows Jessop \cite{J} in proving the following statement.
\begin{proposition} \label{prop:sing}
Let $u=(u_0,\dots,u_5) \in U$. The surface $\mathcal{X}_u$ is singular
if and only if $$(u_0^3+u_1^3+u_2^3+u_3^3+u_4^3+u_5^3) \prod_{i < j}
(u_i+u_j)=0.$$ In other words, the subset of $U$ parameterizing the
singular invariant quartics is the 3-fold $$S_3 \cup \bigcup_{t \in
\mathcal{T}} T_tS_3.$$
\end{proposition}
It is shown in \cite{H} and \cite{J} that a general point on the Segre
cubic corresponds to an invariant surface with exactly 16 nodes which
constitute an orbit under the Heisenberg group. Such a surface
acquires Kummer's $16_6$ configuration of nodes and tropes and is a
Kummer surface associated to some Abelian surface.
\begin{proposition} \label{prop:Kummer}
For generic $u \in S_3$, $\mathcal{X}_u$ is a Kummer surface. 
\end{proposition}
For $u \in U$, we say that $\mathcal{X}_u$ is a \emph{quadric of
  multiplicity two}, if its defining polynomial is a square of a
degree two polynomial. By \emph{invariant tetrahedron} we understand a
Heisenberg invariant quartic surface which is a union of four planes.
One checks by direct computation that there are exactly 10 quadrics of
multiplicity two and 15 invariant tetrahedra in the family. These
correspond to the $10$ nodes $\mathcal{Q}$ of $S_3$ and the 15 points
$\mathcal{T} \subset S_3$. The quadrics of multiplicity two are called
the \emph{fundamental quadrics}. There is an explanation of the
fundamental quadrics and invariant tetrahedra in terms of the
symplectic structure on $H \cong (\mathbb{F}_2)^4$. This vector space
has 35 planes, each consisting of 4 points. If $g,h \in H_{2,2}$ span
a plane in $H$, then that plane will be called isotropic if $g$ and
$h$ commute and non-isotropic if they anticommute. Of the 35 planes, 15
are isotropic and 20 are non-isotropic. For an isotropic plane
$\{1,g,h,gh\}$, the three pairs of fix lines belonging to $g,h$ and
$gh$ are the edges of an invariant tetrahedron. Similarly the
non-isotropic planes explain the fundamental quadrics. A plane is
isotropic if and only if it is equal to its orthogonal complement and
therefore the non-isotropic planes form 10 pairs of orthogonal planes.
For each such pair $P_1,P_2$, the 6 fix lines belonging to elements of
$P_1$ are skew and likewise for $P_2$. Moreover, the 6 fix lines
belonging to elements of $P_1$ intersect all the 6 fix lines belonging
to elements of $P_2$. All 12 lines lie on a fundamental quadric, with
6 lines in one ruling and 6 lines in the other.

The dual variety of $S_3$ is a quartic hypersurface in $\PP^{4}$
called the \emph{Igusa quartic}, which we denote by $I_4$. The
singular locus of $I_4$ consists of 15 lines. One way to view $I_4$ is
to look at the linear system of invariant quartics. Recall the basis
$\{g_0,\dots,g_4\}$ of this system introduced above. Because the
system does not have a base point we have a morphism $$\alpha:\PP^{3}
\rightarrow \PP^{4}:q \mapsto (g_0(q),\dots,g_4(q)).$$ The image of
$\alpha$ is $I_4$ \cite{Ba}. For a hyperplane $L \subset \PP^{4}$ we
have a corresponding Heisenberg invariant quartic surface
$X_L=\alpha^{-1}(L)$. In the following proposition we we will call a
surface $W$ a quotient by $H$ of $X_L$ if there is a morphism $X_L
\rightarrow W$ whose fibers are exactly the orbits of $H$.
\begin{proposition} \label{prop:Igusa}
For any hyperplane $L \subset \PP^{4}$, $L \cap I_4$ is a quotient of
$X_L=\alpha^{-1}(L)$ by the Heisenberg group with quotient map
$\alpha: X_L \rightarrow L \cap I_4$.
\end{proposition}

The major part of \propref{prop:Igusa} follows from the discussion in
\cite{Ba} Chapter VII page 209. More precisely it is shown there that
if $p=(x,y,z,w) \in \PP^3$ is such that $xyzw \neq 0$ then
$\alpha^{-1}(\alpha(p))$ has at most 16 elements. Thus, if in addition
we assume that $p$ is not on any fix line (that is the orbit of $p$
has 16 elements), we have that $\alpha^{-1}(\alpha(p))$ is exactly the
orbit of $p$ since clearly $\alpha$ is constant on the orbits. It is
straightforward to verify the remaining special cases where some
coordinate of $p$ is zero or $p$ is on a fix line.

A generic hyperplane $L \subset \PP^{4}$ intersects $I_4$ in a quartic
surface with 15 singular points and 10 tropes. The singular points
come from the intersections with the singular lines of $I_4$. On the
quotient $X_L/H$, there are 15 singularities that come from points
with nontrivial stabilizers in $H$. Namely, for any of the 15
non-identity elements of $H$, the two fix lines intersect $X_L$ in an 8
point orbit that maps to one singular point on $X_L/H$ by the quotient
map. Now, if $L$ is a generic tangent space to $I_4$, then $I_4 \cap
L$ has an additional singularity at the point of tangency and $I_4
\cap L$ is a Kummer surface. In this way $I_4$ is a compactification
of a moduli space of principally polarized Abelian surfaces with a
level-2 structure \cite{vdG,H}.

Let us now look at one more $ \mathfrak{S}_6$ invariant threefold in
the parameter space $U \cong \PP^{4}$, the \emph{Nieto quintic}
$N_5$. This is the Hessian variety to $S_3$, that is the zeros of the
determinant of the Hessian matrix of the polynomial defining
$S_3$. The singular locus of $N_5$ consists of 20 lines together with
the 10 nodes of $S_3$. With our choice of coordinates, viewing
$\PP^{4} \subset \PP^{5}$, $N_5$ is defined by the equations
$$\sum_{i=0}^5 u_i = \sum_{i=0}^5 \frac{1}{u_i}=0,$$ where
$\sum_{i=0}^5
\frac{1}{u_i}=u_1u_2u_3u_4u_5+u_0u_2u_3u_4u_5+\dots+u_0u_1u_2u_3u_4$. The
Nieto quintic and its connection to Heisenberg invariant quartics
which contain lines is studied in \cite{BN} where the following is
proved (\cite{BN} Sections 7 and 8):
\begin{proposition} \label{prop:nietoquintic}
The locus in $U$ parameterizing Heisenberg invariant quartics which
contain a line is equal to $N_5$ together with the 10 tangent cones to
the isolated singular points of $N_5$ (the 10 nodes of $S_3$).
\end{proposition}

\section{K3 surfaces, lattices and moduli} \label{sec:moduli}
This section consists of background material on lattices, K3 surfaces,
Picard groups and moduli spaces. In particular, we state some well
known results on moduli spaces of K3 surfaces and relations to Picard
groups, references on this topic include \cite{Do,DK,Hu,Huy,S,S2}.

A lattice $M$ is a finitely generated free Abelian group equipped with
an integral symmetric bilinear form. An important invariant of $M$ is
the discriminant, denoted $\discr{M}$, which is equal to the
determinant of any matrix representing the bilinear form. We shall
deal only with non-degenerate lattices, that is we assume that the
discriminant is non-zero. If the discriminant is equal to $\pm 1$, the
lattice is called unimodular. If $N \subseteq M$ is a sublattice of
full rank, then $\discr{N}/\discr{M}=[M:N]^2$, where $[M:N]$ is the
index of $N$ in $M$, see \cite{BHPV} I.2.1. If $N \subseteq M$ is a
sublattice, we use the notation $N^{\perp}$ for the orthogonal
complement of $N$ in $M$. A lattice is called even if the $\ZZ$ valued
quadratic form associated to the bilinear form takes on even values
only. There is an induced bilinear form on the real vector space $M
\otimes_{\ZZ} \mathbb{R}$ which may be diagonalized in such a way that
only $\pm 1$ appear on the diagonal. The signature of a lattice $M$ is
a pair of numbers $(a,b)$, where $a$ is the number of positive entries
in such a diagonal matrix and $b$ the number of negative entries. The
sum of $a$ and $b$ is the rank of $M$. A sublattice $N$ of a lattice
$M$ is called primitive if $M/N$ is torsion free. An isomorphism of
lattices is a group isomorphism that also respects the bilinear form.

For a scheme $X$, the Picard group is the group of isomorphism classes
of line bundles on $X$, and is denoted $\pic{X}$. If $X$ is a smooth
variety over an algebraically closed field, then $\pic{X}$ is
isomorphic to the group of Weil divisors on $X$ modulo linear
equivalence. For a smooth complex projective variety $X$, let
$\piczero{X}$ denote the subgroup of $\pic{X}$ corresponding to
divisors algebraically equivalent to 0 and define the N\'{e}ron-Severi
group by $\NS{X}=\pic{X}/\piczero{X}$. The N\'{e}ron-Severi group is a
finitely generated Abelian group and its rank is called the Picard
number of $X$. Now suppose that $X$ is a complex projective K3
surface, which means that $X$ is a smooth simply connected projective
complex variety of dimension 2, whose canonical bundle is trivial. In
this case $\piczero{X}=\{0\}$ and thus $\NS{X}=\pic{X}$ and $\pic{X}$
is free and finitely generated. It is well known that $1 \leq \rho (X)
\leq 20$, where $\rho{(X)}=\rank{\pic{X}}$ is the Picard number of
$X$. The second cohomology group $H^2(X,\ZZ)$ is also free and of rank
22, and the cup product defines a bilinear form giving it the
structure of a lattice. Abstractly, this lattice is isomorphic to the
so called K3 lattice $$L=U \oplus U \oplus U \oplus (-E_8)\oplus
(-E_8).$$ Here, $U$ is the hyperbolic plane, a rank 2 lattice with
bilinear form given by the matrix
\begin{displaymath}
\left(
\begin{array}{cc}
0 & 1 \\
1 & 0
\end{array}
\right),
\end{displaymath}
and $E_8$ is the lattice of rank 8 whose bilinear form is given by the
Cartan matrix of the root system $E_8$. The K3 lattice is even,
unimodular and of signature $(3,19)$. The Picard group of $X$ may be
viewed as a non-degenerate sublattice of $H^2(X,\ZZ)$ with signature
$(1,\rho{(X)}-1)$, where the bilinear form on $\pic{X}$ is the
intersection product. Since the lattice $L$ is even, so is
$\pic{X}$. The orthogonal complement of $\pic{X}$ in $H^2(X,\ZZ)$ is
called the transcendental lattice and is denoted by $T_X$.

We now turn to a particular class of K3 surfaces, namely smooth
quartic surfaces in $\PP^3$. A homogeneous quartic polynomial in 4
variables has 35 coefficients and therefore $\PP^{34}$ parameterizes
quartic surfaces in $\PP^3$. Consider the open set $O \subseteq
\PP^{34}$ that parameterizes smooth surfaces. The group $\PGL{4}{\CC}$
acts on $O$ by means of isomorphisms of $\PP^3$ and we will use
$\mathcal{M}_2$ to denote the GIT quotient with respect to this
action, see \cite{Huy} Chapter 5 Example 2.5. Thus $\mathcal{M}_2$ is
the (coarse) moduli space of smooth quartic surfaces in $\PP^3$
parameterizing such surfaces up to projective transformations. It is
an irreducible quasi projective variety of dimension 19.

A marking of a K3 surface $X$ is an isomorphism of lattices $L \cong
H^2(X,\ZZ)$. Below we will make use of the (fine) moduli space of
marked polarized K3 surfaces of degree 4. This is a complex manifold
which we will denote by $\mathcal{N}_2$, see \cite{Huy} Chapter 6
3.4. For $m \in \mathcal{N}_2$, let $\rho(m)$ denote the Picard number
of the surface corresponding to $m$ and consider the \emph{higher
  Noether-Lefschetz loci},
$$\NL{k}=\{m \in \mathcal{N}_2:\rho(m) \geq k\},$$ where $k$ is an
integer $1\leq k \leq 20$. The Noether-Lefschetz locus $\NL{k}$ may be
written as a countable union of analytic subvarieties of
$\mathcal{N}_2$ of dimension $20-k$, see \cite{Huy} Chapter 17
1.3.

\begin{definition}
We say that a condition holds for a \emph{very generic} point, or
\emph{very general} point, of a variety $Y$ if there is a countable
union $K=\bigcup_{i \in \ZZ} K_i$ of proper Zariski-closed subsets
$K_i \subset Y$, such that the condition holds for all points $y \in Y
\setminus K$.
\end{definition}

Let $O \subset \PP^{34}$ be the subset parameterizing smooth quartic
surfaces, let $Z \subseteq O$ be a subvariety and $\mathcal{Y}
\rightarrow Z$ the corresponding family of surfaces. We use
$\rho_0(Z)$ to denote the minimal Picard number of any fiber in the
family.

\begin{lemma} \label{lemma:dimension}
If $Z \subseteq O \subseteq \PP^{34}$ is a subvariety of $O$ such that
$Z \rightarrow \mathcal{M}_2$ is generically finite to one, then
$\dim{Z} \leq 20-\rho_0(Z)$.
\end{lemma}
\begin{proof}
  Let $k = \rho_0(Z)$. For a general smooth point $p \in Z$ there is
  an open analytic neighborhood $V \subseteq O$ of $p$ such that
  $V\cap Z \rightarrow \mathcal{M}_2$ is injective. By shrinking $V$
  we may assume that $B=V\cap Z$ is smooth, connected and simply
  connected. The family of surfaces with base $B$ may be marked by
  choosing a marking of some fiber, which yields a canonical marking
  of all fibers, see \cite{Huy} Chapter 6 2.2. This gives a
  holomorphic map $\phi: B \rightarrow \mathcal{N}_2$ which is also
  injective since if $\phi(b)=\phi(b')$ then the fibers over $b$ and
  $b'$ are projectively equivalent. Since $\phi(B) \subseteq \NL{k}$
  and $\NL{k}$ is a countable union of $(20-k)$-dimensional analytic
  subvarieties of $\mathcal{N}_2$, $\dim{Z}=\dim{B} \leq 20-k$.
  \end{proof}

\section{The Picard number} \label{sec:picard}
We will show that the Picard number of a very general Heisenberg
invariant quartic is 16. A corollary to this is that any invariant
divisor class is a multiple of the hyperplane section. We start with
some generalities on symplectic group actions on K3 surfaces.

If a group $G$ acts on a K3 surface $X$, then $G$ acts on $H^2(X,\ZZ)$
and the action restricts to the sublattices $\pic{X}$ and $T_X$. We
will write this action $g^{\ast}l$ for $g \in G$ and $l \in
H^2(X,\ZZ)$. Let $$H^2(X,\ZZ)^G = \{l \in H^2(X,\ZZ): g^{\ast}l=l, \;
\textrm{for all} \; g \in G\},$$ and $\pic{X}^G=H^2(X,\ZZ)^G \cap
\pic{X}$. Since the canonical bundle of a K3 surface $X$ is trivial,
$X$ has a unique non-zero holomorphic 2-form $\omega$, up to
rescaling. Hence for $f\in \aut{X}$, $f^* \omega=\lambda \omega$,
where $\lambda \in \CC$. If $\lambda=1$, then $f$ is called
symplectic. In the following proposition, the first claim is due to
Nikulin \cite{Nik} and the second claim is due to Mukai \cite{M}.

\begin{proposition} \label{prop:NikulinMukai}
Let $G$ be a finite group of symplectic automorphisms of a K3 surface
$X$ and let $\Omega_G=(H^2(X,\ZZ)^G)^{\perp}$ be the orthogonal
complement of $H^2(X,\ZZ)^G$ in $H^2(X,\ZZ)$. For $g \in
G$, $g \neq 1$, let $f(g)$ be the number of fixed points of $g \in G$
and put $f(1)=24$. Then
\begin{enumerate}
\item $T_X \subseteq H^2(X,\ZZ)^G$ and $\Omega_G \subseteq
\pic{X}$,
\item $\rank{H^2(X,\ZZ)^G}+2=\frac{1}{|G|}\sum_{g \in G}f(g)$.
\end{enumerate}
\end{proposition}

The generic member of our family is a smooth quartic in $\PP^{3}$ and
hence a K3 surface. Since the action of the Heisenberg group is
induced by a representation $H_{2,2} \rightarrow \SL{4}{\CC}$ and the
action of $H_{2,2}$ preserves the defining polynomials of our
surfaces, the action of $H$ is symplectic by \cite{M} Lemma 2.1.

\begin{corollary} \label{coro:ranks}
If $X \subset \PP^{3}$ is a generic Heisenberg invariant quartic, then
\begin{enumerate}
\item $\rank{H^2(X,\ZZ)^H}=7$,
\item $16 \leq \rho (X)$.
\end{enumerate}
\end{corollary}
\begin{proof}
Let $g \in H$, $g\neq 1$. The fixed points of $g$ form two skew lines in
$\PP^{3}$, and hence $g$ has 8 fixed points on $X$.  By
\propref{prop:NikulinMukai}, $T_X \subseteq H^2(X,\ZZ)^H$ and
$\rank{H^2(X,\ZZ)^H}=7$. Assume that $\rank{T_X}=7$. The
hyperplane class $h \in \pic{X}$ is invariant and thus there is an
integer $n>0$ such that $nh \in T_X$. Since $T_X$ is the orthogonal
complement of $\pic{X}$, $0=(nh)h=4n$, a contradiction. Hence
$\rank{T_X} \leq 6$. It follows that $16 \leq \rank{\pic{X}}$, since
$\rank{\pic{X}}+\rank{T_X}=\rank{H^2(X,\ZZ)}=22$.
\end{proof}

For homogeneous $F \in \CC [x,y,z,w]$, let $X=\{F=0\} \subset \PP^3$
and consider the Hessian surface of $X$,
$$\textrm{Hess}(X)=\{x \in \PP^{3}: \det{(\textrm{Hess}(F)}(x))=0\},$$
where $\textrm{Hess}(F)$ is the Hessian matrix of second derivatives
of $F$. Let $T \in \GL{4}{\CC}$ and denote the induced automorphism of
$\PP^{3}$ also by $T$. Let $Y=T(X)$ and let $(T^{-1})^t$ denote the
transpose of $T^{-1}$. Note that for $p \in \CC^{4}$,
$$\textrm{Hess}(F \circ T^{-1})(p)=(T^{-1})^t \cdot
\textrm{Hess}(F)(T^{-1}(p)) \cdot T^{-1}.$$ It follows that $T$ takes
the Hessian surface of $X$ to the Hessian surface of $Y$, with
preserved multiplicities of irreducible components. In particular, the
Hessian surface of a Heisenberg invariant quartic is a Heisenberg
invariant octic.

\begin{example} \label{ex:fermat}
Let $u=(1,1,1,1,-2,-2) \in U$. We will look at linear transformations
of the Fermat quartic $X = \mathcal{X}_u=\{(x,y,z,w) \in
\PP^{3}:x^4+y^4+z^4+w^4=0\}$ for the purpose of applying
\lemmaref{lemma:dimension}.

There are 15 surfaces in the family that are projectively equivalent
to $X$ corresponding to the $\mathfrak{S}_6$-orbit of
$(1,1,1,1,-2,-2)$, and we shall see that there are no others. The
Hessian surface of the Fermat quartic is the zeros of $x^2y^2z^2w^2$,
an invariant tetrahedron of double multiplicity. Note that the
invariant tetrahedron defined by $xyzw$ corresponds to the point
$(0,0,0,0,1,-1) \in U$. Suppose that $T \in \GL{4}{\CC}$ takes $X$ to
some Heisenberg invariant quartic $Y$. By the discussion above, the
Hessian surface of $Y$ must be one of the 15 invariant tetrahedra,
counted with multiplicity 2. These correspond to the
$\mathfrak{S}_6$-orbit of $(0,0,0,0,1,-1)$. Let $S \in \SL{4}{\CC}$ be
a transformation in the normalizer of $H_{2,2}$ in $\SL{4}{\CC}$ which
takes $\textrm{Hess}(Y)$ to $\textrm{Hess}(X)$. Then $S \circ T:\PP^3
\rightarrow \PP^3$ induces an automorphism of $\textrm{Hess}(X)$ and
therefore corresponds to a permutation of the coordinates $(x,y,z,w)$
possibly combined with a rescaling of the axes. Since $S(T(X))$ is an
invariant quartic, $S(T(X))=X$ and $T(X)=S^{-1}(X)$. Therefore $T(X)$
is in the $\mathfrak{S}_6$-orbit of $(1,1,1,1,-2,-2)$.
\end{example}

\begin{theorem} \label{thm:picardnumber}
A very generic Heisenberg invariant quartic $X \subset \PP^{3}$ has Picard number 16.
\end{theorem}
\begin{proof}
  By the second claim of \cororef{coro:ranks}, $\rho (X) \geq 16$. Let
  $Z \subset \PP^{34}$ be the subset parameterizing smooth Heisenberg
  invariant quartics. The quotient map $Z \rightarrow \mathcal{M}_2$
  is generically finite to one by upper-semi-continuity of the
  dimension of fibers and \exref{ex:fermat}. By
  \lemmaref{lemma:picinjections}, $\rho (X)=\rho_0(Z)$ and using
  \lemmaref{lemma:dimension} we get $\rho (X) \leq 20-4=16$.
\end{proof}

\begin{corollary} \label{coro:picinv}
For very generic $u \in U$ and $X=\mathcal{X}_u$, $\pic{X}^H=\ZZ h$,
where $h \in \pic{X}$ is the class of the hyperplane section. In
particular, every invariant curve on $X$ is a complete intersection
between $X$ and another surface in $\PP^{3}$.
\end{corollary}
\begin{proof}
Since
$\rank{\pic{X}}=16$, $$\rank{T_X}=\rank{H^2(X,\ZZ)}-\rank{\pic{X}}=6.$$
Because $T_X \cap \pic{X}=\{0\}$, and $T_X$ is contained in the rank 7
lattice $H^2(X,\ZZ)^H$, it follows that $\pic{X}^H$ has rank at most
1. Clearly, the class $h$ is invariant under $H$ and thus
$\rank{\pic{X}^H}=1$. Let $D$ be a generator of $\pic{X}^H$ and let $n
\in \ZZ$ be such that $h=nD$. Since $4=h^2=n^2D^2$, and $D^2$ is even
since $\pic{X}$ is an even lattice, it follows that $n=\pm 1$.
\end{proof}

\section{Conics on the invariant surfaces} \label{sec:conics}

\begin{definition} \label{def:tropes}
Let $X$ be a quartic surface in $\PP^{3}$. We say that a plane $L$ in
$\PP^{3}$ is a \emph{trope} of $X$ if $X \cap L$ is an irreducible
conic counted with multiplicity two.
\end{definition}
\begin{lemma} \label{lemma:tropesing} 
A quartic surface $X \subset \PP^{3}$ which has a trope $L$ is
necessarily singular.
\end{lemma} 
\begin{proof} 
Let $(x,y,z,w)$ be coordinates on $\PP^{3}$. Change coordinates so
that $L=\{x=0\}$. Then $X$ is defined by $xF(x,y,z,w)+(G(y,z,w))^2=0$,
for some cubic polynomial $F$ and quadratic polynomial $G$. Then the
set $M = \{x=0\} \cap \{F=0\} \cap \{G=0\}$ is nonempty and inside the
singular locus of $X$.
\end{proof}

Below we shall make use of the following fact: for generic $u \in S_3$,
the Kummer surface $\mathcal{X}_u$ is uniquely determined by any of
its nodes as well as any of its tropes. To establish this, let
$F_{\lambda}(x,y,z,w)$ be a Heisenberg invariant polynomial depending
on the parameters $\lambda=(A,B,C,D,E)$ as in (\ref{eq:inv_pol}). For
a general $p \in \PP^{3}$ there is a unique $u \in U$ such that
$\mathcal{X}_u$ is singular at $p$. To see this, it is enough to check
that there is a point $p=(x,y,z,w) \in \PP^3$ such that the system of
linear equations
$$\frac{\partial F_{\lambda}}{\partial x}(p)=\frac{\partial
F_{\lambda}}{\partial y}(p)=\frac{\partial F_{\lambda}}{\partial
z}(p)=\frac{\partial F_{\lambda}}{\partial w}(p)=0$$ has a unique
solution $(A,B,C,D,E) \in \PP^4$ (this is true for example if
$p=(1,2,3,4)$). In fact, there exists a unique invariant quartic
singular at $p \in \PP^3$ exactly when $p$ is not on a fix line
\cite{BN}. Using $\mathfrak{S}_6$-invariance one checks that the set
of $u \in U$ such that $\mathcal{X}_u$ has a singular point that lies
on some fix line is equal to the union of the $15$ tangent spaces $T_t
S_3$ where $\mathcal{X}_t$ is an invariant tetrahedron. We conclude
that a generic Heisenberg invariant Kummer surface is determined by
any of its nodes. The corresponding statement for tropes is clear once
we see that $p \in \PP^{3}$ is a node precisely when the plane
$p^{\ast} \in {\PP^{3}}^{\ast}$ with the same coordinates as $p$ is a
trope. Let $p$ be a node, and note that we may assume $p$ to be a
generic point in $\PP^3$. As in \cite{H} Chapter I $\S 3$, there are 6
points in the orbit of $p$ that lie in $p^{\ast}$. It follows that the
6 points are multiple points of the curve of intersection between
$p^{\ast}$ and $\mathcal{X}_u$. This curve must then either contain a
line as an irreducible component or be a conic of multiplicity
two. The first case is excluded by \propref{prop:nietoquintic}. The
set of 16 tropes is hence the orbit of $p^{\ast}$.

\begin{theorem} \label{thm:conics}
A generic invariant quartic $\mathcal{X}_u$ contains at least 320 smooth conics.
\end{theorem}
\begin{proof}
Pick $u \in U$ generic and let $q$ be a node of the Segre cubic
$S_3$. The line through $u$ and $q$ intersects $S_3$ in one additional
point $p$, let $K(q) = \mathcal{X}_p$. Since $u$ is generic, $K(q)$ is
a Kummer surface by \propref{prop:Kummer}, and thus has 16 tropes
which form an orbit under $H$. Let $T$ be a trope of $K(q)$ and note
that $\mathcal{X}_q$ is a quadric of multiplicity two. The polynomial
defining $\mathcal{X}_u \cap T$, in homogeneous coordinates on $T$, is
a linear combination of squares, and thus reducible. The generic
member of the family does not contain any line
(\propref{prop:nietoquintic}), and since $\mathcal{X}_u$ is smooth nor
does it have a trope (\lemmaref{lemma:tropesing}). We conclude that
$\mathcal{X}_u \cap T$ is a union of two smooth conics.

For two different nodes $q_1$ and $q_2$ of $S_3$ the corresponding
Kummer surfaces $K(q_1)$ and $K(q_2)$ are different, since we may
assume that $u$ is not on the line through $q_1$ and $q_2$. Since a
generic Heisenberg invariant Kummer surface is determined by any of
its tropes, all tropes of the Kummer surfaces $K(q)$, for all of the
10 nodes $q$ of $S_3$, are different. Since two different planes
cannot have a smooth conic in common, we conclude that there are at
least $2 \cdot 16 \cdot 10=320$ smooth conics on $\mathcal{X}_u$.
\end{proof}

For a general invariant quartic $X$, we will refer to the conics in
\thmref{thm:conics} as \emph{the 320 conics} on $X$. As we show in
\propref{prop:configuration}, these conics are the only conics on a
very general surface in the family.

\section{Invariant surfaces containing lines} \label{sec:lines}
The generic Heisenberg invariant quartic does not contain any
line. The invariant surfaces that contain a line are parameterized by
the Nieto quintic $N_5$ and the ten tangent cones to the isolated
singularities of $N_5$ \cite{BN}. In this section we will look at the
configuration of lines on a Heisenberg invariant quartic that
corresponds to a general point on $N_5$. This will be used to compute
the intersection matrix of the 320 smooth conics on a very generic
invariant quartic. General references for this section are
\cite{BB,BN,B} and in particular the proof of
\propref{prop:nootherconics} follows arguments made in \cite{BB} and
\cite{B}.

Let $(p_{01},p_{02},p_{03},p_{12},p_{13},p_{23})$ be the Pl\"ucker
coordinates on $\PP^5$. The Pl\"ucker embedding identifies the
Grassmannian of lines in $\PP^3$ with the quadric in $\PP^5$ defined
by the Pl\"ucker relation
$$p_{01}p_{23}-p_{02}p_{13}+p_{03}p_{12}=0.$$ We will use so-called
Klein coordinates $(x_0,\dots,x_5)$, which are defined in terms of
Pl\"ucker coordinates by
\begin{center}
\begin{tabular}{lll}
$x_0=p_{01}-p_{23}$, & $x_2=p_{02}+p_{13}$, & $x_4=p_{03}-p_{12}$, \\
$x_1=i(p_{01}+p_{23})$, & $x_3=i(p_{02}-p_{13})$, & $x_5=i(p_{03}+p_{12})$,
\end{tabular}
\end{center}
where $i^2=-1$. In Klein coordinates the Pl\"ucker relation reads
$$x_0^2+x_1^2+x_2^2+x_3^2+x_4^2+x_5^2=0,$$ and the condition that a
line with coordinates $(x_0,\dots,x_5)$ is coplanar with a line with
coordinates $(y_0,\dots,y_5)$ is
\begin{equation} \label{eq:coplanar}
x_0y_0+x_1y_1+x_2y_2+x_3y_3+x_4y_4+x_5y_5=0.
\end{equation}
The Heisenberg group acts on the space of lines in $\PP^3$ and in
Klein coordinates the action is given neatly by sign changes in
accordance with the table
\begin{center}
\begin{tabular}{cccccccc}
 &\vline& $x_0$ & $x_1$ & $x_2$ & $x_3$ & $x_4$ & $x_5$ \\
\hline 
$\sigma_1$ &\vline& $-$ & $+$ & $-$ & $-$ & $+$ & $-$ \\
$\sigma_2$ &\vline& $-$ & $-$ & $+$ & $-$ & $-$ & $+$ \\
$\tau_1$ &\vline& $+$ & $+$ & $-$ & $-$ & $-$ & $-$ \\
$\tau_2$ &\vline& $-$ & $-$ & $+$ & $+$ & $-$ & $-$
\end{tabular}
\end{center}
where $\sigma_1,\sigma_2,\tau_1,\tau_2$ are the generators from
\secref{sec:Hgroup}.

In \cite{BN} it is shown that for a generic $u \in N_5$,
$\mathcal{X}_u$ contains exactly 32 lines. In the same paper it is
proved that such a surface is a desingularized Kummer surface coming
from an Abelian surface $A$ with a $(1,3)$-polarization. The
Klein coordinates for any of the 32 lines satisfy the condition
$$\frac{1}{x_0^2}+\frac{1}{x_1^2}+\frac{1}{x_2^2}+\frac{1}{x_3^2}+\frac{1}{x_4^2}+\frac{1}{x_5^2}=0,$$
by which we mean that
$$x_1^2x_2^2x_3^2x_4^2x_5^2+x_0^2x_2^2x_3^2x_4^2x_5^2+ \dots
+x_0^2x_1^2x_2^2x_3^2x_4^2=0.$$ The 32 lines constitute two orbits
under the Heisenberg group, each containing 16 lines. There is an
involution relating the two orbits: if $(x_0,\dots,x_5)$ are Klein
coordinates for a line in one of the orbits, then
$$(-\frac{1}{x_0},\frac{1}{x_1},\frac{1}{x_2},\frac{1}{x_3},\frac{1}{x_4},\frac{1}{x_5})$$
are Klein coordinates for a line in the other orbit. The sixteen lines
in any of the two orbits are mutually disjoint, in fact one of the
orbits is the union of the exceptional divisors coming from blowing up
the singular Kummer surface in its 16 nodes. The configuration of
lines is called a $32_{10}$ since any line in one orbit intersects
exactly 10 lines in the other orbit. There are thus exactly 160
reducible conics on the surface. If two lines, one from each orbit,
are coplanar then their \emph{complement} on the surface is an
irreducible conic. In other words, the plane that contains the
reducible conic intersects the surface in the union of two lines and
an irreducible conic. This gives rise to 160 irreducible conics on
$\mathcal{X}_u$.
\begin{proposition} \label{prop:nootherconics}
For a very general $u \in N_5$ there are no other irreducible conics
on $\mathcal{X}_u$ except for the 160 complements to reducible conics.
\end{proposition}
\begin{proof}
Let $X=\mathcal{X}_u$. The surface $X$ is a desingularized Kummer
surface of an Abelian surface $A$ admitting a $(1,3)$-polarization. In
fact, there is an open subset of the moduli space of Abelian surfaces
which admit a $(1,3)$-polarization together with the choice of a
level-2 structure which is naturally an unbranched double cover of the
open subset of $N_5$ parameterizing smooth surfaces \cite{BN}. Since
$u \in N_5$ is very general we may assume that $A$ has Picard number 1
(see \cite{BL} Section 15.1). Let $e_1,\dots,e_{16}$ denote the
halfperiods of $A$, that is $e_i$ is a 2-torsion point or the identity
element. Let $\tilde{A}$ denote the blow-up of $A$ in
$\{e_1,\dots,e_{16}\}$. We have morphisms $$A \xleftarrow{\alpha}
\tilde{A} \xrightarrow{\gamma} X,$$ where $\gamma$ is a double cover
branched over the 16 lines $L_1,\dots,L_{16}$ in one of the orbits and
$\alpha$ is the blow-up map \cite{BN}. Let $E_i = \alpha^{-1}(e_i)$,
$i=1,\dots,16$. Then $E_i \cdot E_j=0$ if $i \neq j$, $E_i^2=-1$ and
up to reordering $\gamma (E_i)=L_i$. Now, as is explained in
\cite{BN}, if $M$ is the line bundle corresponding to the sheaf
$\mathcal{O}_X(1)$, then the line bundle $\gamma^*(M) \otimes
\mathcal{O}_{\tilde{A}}(\sum_{i=1}^{16} E_i)$ on $\tilde{A}$ descends
to a line bundle on $A$ which defines a polarization of type
$(2,6)$. Further, there is a symmetric line bundle $\theta$ on $A$ of
type $(1,3)$ such that $\alpha^*(\theta \otimes \theta)=\gamma^*(M)
\otimes \mathcal{O}_{\tilde{A}}(\sum_{i=1}^{16} E_i)$, see
\cite{BN}. Let $T$ denote the divisor class in the N\'{e}ron-Severi
group of $A$ that corresponds to $\theta$ and let $H_X$ be the
hyperplane class of $X$. Then we have that $T^2=2 d_1 d_2$ where
$(d_1,d_2)=(1,3)$ \cite{Mum}. Thus $T^2=6$.

Now let $C \subseteq X$ be an irreducible conic different from the
complements of the reducible conics. By the adjunction formula, we
have that $C^2=-2$. Let $F=\alpha_*(\gamma^*(C))$. Since $T^2=6$ is
square free, $T$ generates the N\'{e}ron-Severi group of $A$ and we
may write $F=dT$ for some integer $d$. Now let $m_i=C \cdot
L_i=\gamma^*(C) \cdot E_i$ for $i=1,\dots,16$. Then
$\alpha^*(F)=\gamma^*(C)+\sum_{i=1}^{16}m_iE_i$. Hence
$6d^2=(dT)^2=F^2=\alpha^*(F)^2=(\gamma^*(C)+\sum_{i=1}^{16}m_iE_i)^2$,
and therefore $$6d^2= \gamma^*(C)^2+2\sum_{i=1}^{16}m_i\gamma^*(C)
\cdot E_i+\sum_{i=1}^{16}m_i^2E_i^2=
2C^2+2\sum_{i=1}^{16}m_i^2-\sum_{i=1}^{16}m_i^2.$$  In conclusion,
\begin{equation} \label{eq:F_squared}
6d^2=\sum_{i=1}^{16}m_i^2-4.
\end{equation}
Furthermore, $$12d=2T \cdot dT=2T \cdot F=\alpha^*(2T)\cdot
\alpha^*(F)=(\gamma^*(H_X)+\sum_{i=1}^{16}E_i)\cdot
(\gamma^*(C)+\sum_{i=1}^{16}m_iE_i),$$ and hence
$$12d=\gamma^*(H_X)\cdot \gamma^*(C)+\sum_{i=1}^{16}m_iE_i\cdot
\gamma^*(H_X)+\sum_{i=1}^{16}E_i \cdot
\gamma^*(C)+\sum_{i=1}^{16}m_iE_i^2,$$
which implies that
\begin{equation} \label{eq:2TF}
12d=4+\sum_{i=1}^{16}m_i+\sum_{i=1}^{16}m_i-\sum_{i=1}^{16}m_i=4+\sum_{i=1}^{16}m_i.
\end{equation}
Since $L_i$ and $C$ are not coplanar, $m_i=0$ or $m_i=1$ for all
$i$. It follows by (\ref{eq:2TF}) that $d=1$. Note also that
$m_i=m_i^2$ for all $i$. But then (\ref{eq:F_squared}) gives
$\sum_{i=1}^{16}m_i=10$ and (\ref{eq:2TF}) gives
$\sum_{i=1}^{16}m_i=8$, a contradiction.
\end{proof}

Let $u \in N_5$ be very generic and put $X=\mathcal{X}_u$. As long as
all the Klein coordinates $(x_0,\dots,x_5)$ of a line $l$ are
non-zero, the matter of whether some Heisenberg translate of $l$
intersects some Heisenberg translate of the image of $l$ under the
involution does not depend on $l$. From the description above of the
action of the Heisenberg group on the Grassmannian of lines in
$\PP^3$, the condition (\ref{eq:coplanar}), and the involution
relating the two sets of 16 lines on $X$, it is straightforward to
compute the intersection matrix of the 160 reducible conics on
$X$. Since the 160 irreducible conics on $X$ are coplanar with
reducible conics, the intersection matrix of all 320 conics on $X$ is
readily deducible from the intersection matrix of the 160 reducible
conics. We will not do this in detail but we note in passing that if
$C$ and $D$ are conics on $X$ such that $C$ and $D$ belong to
different orbits under the Heisenberg group $H$, then
\begin{itemize}
\item[] $C \cdot gC=0$ for 6 different $g \in H$, $g\neq 1$,
\item[] $C \cdot gC=2$ for 9 different $g \in H$, $g\neq 1$,
\item[] $C \cdot gD=0$ for 4 different $g \in H$,
\item[] $C \cdot gD=1$ for 8 different $g \in H$,
\item[] $C \cdot gD=2$ for 4 different $g \in H$.
\end{itemize}
We will see in \secref{sec:NS} that the intersection matrix of the
conics on $X$ is the same as the intersection matrix associated to the
320 smooth conics on a very generic Heisenberg invariant quartic.

\begin{remark} \label{rem:16conics}
For later purposes we will consider the intersections of a special set
of 16 reducible conics on $X$. We need to put an order on the set of
all reducible conics on $X$. We will order the elements of the
Heisenberg group as follows. Let $g \in H$, $g=\sigma_1^i \sigma_2^j
\tau_1^k \tau_2^l$ where $0 \leq i,j,k,l \leq 1$ and
$\sigma_1,\sigma_2,\tau_1$ and $\tau_2$ are the generators from
\secref{sec:Hgroup}. Then we order $H$ by interpreting $(i,j,k,l)$ as
a number in binary form that marks the place of $g$. Let $L$ be a line
on $X$. The order on $H$ puts an order on the orbit of $L$ as well as
the orbit of the image of $L$ under the involution. This induces an
order on the set of reducible conics by saying that if line number $a$
in the orbit of $L$ intersects line number $b$ in the other orbit, and
likewise for $a'$ and $b'$ with $(a,b) \neq (a',b')$, then the
reducible conic indexed by $(a,b)$ is prior to the conic indexed by
$(a',b')$ if and only if $a<a'$, or $a=a'$ and $b < b'$. With that
order of rows and columns, let $N$ denote the intersection matrix of
the reducible conics on $X$. Let $M$ be the submatrix of $N$ given by
picking out the following rows and columns:
$$(4,7,21,27,36,50,75,81,88,110,114,128,131,138,141,154).$$
This particular choice of conics has been made because, as we shall
see in \secref{sec:picardgroup}, it defines a sublattice with minimal
discriminant (the conics were found using a computer aided search). Then
\begin{displaymath}
M=\left(
\begin{smallmatrix}
-2&0&2&1&2&2&1&0&0&2&0&1&1&2&1&1\\
0&-2&1&0&2&1&2&0&0&1&2&2&1&2&1&2\\
2&1&-2&0&1&1&2&0&1&1&1&0&0&0&1&2\\
1&0&0&-2&1&0&2&1&2&1&1&1&1&1&1&2\\
2&2&1&1&-2&1&1&2&2&0&1&2&1&1&1&0\\
2&1&1&0&1&-2&1&2&1&0&2&2&1&1&2&1\\
1&2&2&2&1&1&-2&1&0&2&1&1&1&0&1&1\\
0&0&0&1&2&2&1&-2&0&1&1&1&0&1&0&2\\
0&0&1&2&2&1&0&0&-2&2&1&2&0&1&2&1\\
2&1&1&1&0&0&2&1&2&-2&2&2&1&1&1&0\\
0&2&1&1&1&2&1&1&1&2&-2&0&1&1&2&0\\
1&2&0&1&2&2&1&1&2&2&0&-2&2&0&1&2\\
1&1&0&1&1&1&1&0&0&1&1&2&-2&0&0&0\\
2&2&0&1&1&1&0&1&1&1&1&0&0&-2&1&1\\
1&1&1&1&1&2&1&0&2&1&2&1&0&1&-2&1\\
1&2&2&2&0&1&1&2&1&0&0&2&0&1&1&-2
\end{smallmatrix}
\right),
\end{displaymath}
and $\textrm{det}{(M)}=-512$.
\end{remark}

\section{The intersection matrix of the 320 conics} \label{sec:NS} 
In \secref{sec:lines} we determined the configuration of conics
(reducible or irreducible) on an invariant quartic $\mathcal{X}_u$ for
very generic $u \in N_5$. In this section we shall show that this
configuration is the same as the configuration of the 320 irreducible
conics on a very general invariant quartic. That is, the intersection
matrices of the two collections of curves are the same up to
reordering of the curves. The 320 smooth conics vary in a family of
curves whose configuration we can compute in a special case, namely
the case of a surface that contains lines.

In the process we shall also show that there are no more conics on a
very general Heisenberg invariant quartic than the 320 conics of
\thmref{thm:conics}.

\begin{lemma} \label{lemma:picinjections}
Let $O \subset \PP^{34}$ be the open subset parameterizing smooth
quartic surfaces, let $Z \subseteq O$ be a smooth subvariety and
$\mathcal{Y} \rightarrow Z$ the corresponding family of surfaces. For
a very general point $z \in Z$ and any $z' \in Z$ there is an
injection of lattices $$i:\pic{\mathcal{Y}_z} \hookrightarrow
\pic{\mathcal{Y}_{z'}}$$ which respects hyperplane classes.
\end{lemma}
\begin{proof}
The N\'{e}ron-Severi group of the geometric generic fiber of
$\mathcal{Y}$ injects into $\pic{\mathcal{Y}_z}$ for any $z \in Z$
(see \cite{MP} Proposition 3.6 (a)). Moreover, for very general $z \in
Z$, that injection is an isomorphism (see \cite{MP} Proposition 3.6
(a) and Corollary 3.9).  This induces an injection
$i:\pic{\mathcal{Y}_z} \rightarrow \pic{\mathcal{Y}_{z'}}$, for a very
general $z \in Z$ and any $z' \in Z$. This injection respects
intersection numbers (see \cite{Huy} Proposition 2.10 Chapter
17). Also, the following diagram commutes
$$\xymatrix{\pic{\mathcal{Y}} \ar[rd]_{r_2} \ar[rr]^{r_1} & &
  \pic{\mathcal{Y}_{z'}} \\ & \pic{\mathcal{Y}_z} \ar[ru]_{i} &},$$
where $r_1$ is the restriction map $j^*$ with $j: \mathcal{Y}_{z'}
\hookrightarrow \mathcal{Y}$ the inclusion and similarly for $r_2$
(this follows from the construction in \cite{SGA} X 7.8). Consider the
polarizations on $\mathcal{Y}_z$ and $\mathcal{Y}_{z'}$ that define
the embeddings as quartic surfaces in $\PP^3$. For a hyperplane $L
\subset \PP^3$, these polarizations both come from the divisor class
$(\PP^{34} \times L) \cap \mathcal{Y}$ on $\mathcal{Y}$ via the
restriction maps. Hence, $i$ maps the hyperplane class of
$\mathcal{Y}_z$ to the hyperplane class of $\mathcal{Y}_{z'}$.
\end{proof}

\begin{proposition} \label{prop:configuration}
For very general points $u \in U$ and $u' \in N_5$ we have that
\begin{enumerate}
\item the quartic $\mathcal{X}_u$ has exactly 320 conics and no lines,
\item the configuration of irreducible conics on $\mathcal{X}_{u}$ is
  the same as the configuration of conics (reducible or irreducible)
  on $\mathcal{X}_{u'}$.
\end{enumerate}
\end{proposition}
\begin{proof}
Let $u \in U$ and $u' \in N_5$ be such that there are only 320 conics
on $\mathcal{X}_{u'}$, the surfaces $\mathcal{X}_u$ and
$\mathcal{X}_{u'}$ are smooth, $\mathcal{X}_u$ contains the
configuration of 320 smooth conics and there is an injection of
lattices $$i:\pic{\mathcal{X}_u} \hookrightarrow
\pic{\mathcal{X}_{u'}}$$ which respects hyperplane classes. As we have
seen, these conditions are satisfied for very general points $u \in U$
and $u' \in N_5$, see \lemmaref{lemma:picinjections} for the part
about $i$. Let $C \in \pic{\mathcal{X}_u}$ be the class of a
conic. Since $i(C)^2=-2$, either $i(C)$ or $-i(C)$ is effective, see
\cite{BHPV} Chapter VIII Proposition 3.7 (i). But the degree of $i(C)$
is equal to 2, and hence $i(C)$ is effective. Since $i(C)^2=-2$,
$i(C)$ is represented either by an irreducible conic or a reducible
conic. But there are only 320 conics on $\mathcal{X}_{u'}$, and it
follows that there are no more conics on $\mathcal{X}_u$
either. Moreover, since $i$ maps the classes of the 320 smooth conics
on $\mathcal{X}_u$ onto the set of classes of conics on
$\mathcal{X}_{u'}$, we have that the configurations of conics on the
two surfaces coincide.
\end{proof}

\section{The Picard group} \label{sec:picardgroup} In this section we
determine the Picard group of a very general Heisenberg invariant
quartic and show that it is generated by the 320 conics on the
surface.

It has been known since over a hundred years that every curve $C$ on a
general enough Kummer surface $Y$ in $\PP^3$ is such that $C$ counted
with multiplicity 2 is a complete intersection between $Y$ and some
other surface in $\PP^3$, see \cite{H} Chapter XIII. A modern
treatment is given in \cite{GD,GD2}. In particular, every curve on $Y$
has even degree. To be more precise, if $Y \subset \PP^3$ is the
Kummer surface of a principally polarized Abelian surface $A$ with
Picard number 1, which is true for a very general point of the moduli
space of principally polarized Abelian surfaces, then every curve on
$Y$ has even degree (see \cite{GD2} Theorem 4.32). The Kummer surfaces
that appear in the family $\mathcal{X}$ are Heisenberg invariant but
this merely reflects a choice of coordinates on $\PP^3$ (see
\cite{GD2} Lemma 1.52).

\begin{lemma} \label{lemma:odd_degree}
For a very generic Heisenberg invariant quartic surface $X$, no curve
on $X$ has odd degree.
\end{lemma}
\begin{proof}
Let $\textrm{Hilb}_{(d,g)}(\PP^3)$ denote the Hilbert scheme of curves
in $\PP^3$ with Hilbert polynomial $P(x)=dx+(1-g)$. Recall that
$\textrm{Hilb}_{(d,g)}(\PP^3)$ is projective, see \cite{Ko} Chapter I
1.4. Let $\mathcal{Y} \rightarrow \textrm{Hilb}_{(d,g)}(\PP^3)$ be the
universal family. Define an incidence $I_{(d,g)}$ by
$$\textrm{Hilb}_{(d,g)}(\PP^3) \times U \supseteq I_{(d,g)} =
\{(s,u):\mathcal{Y}_s\subset \mathcal{X}_u\},$$ and let $\pi :
\textrm{Hilb}_{(d,g)}(\PP^3) \times U \rightarrow U$ denote the
projection. Now, since for odd $d$ and any $g$, the parameter point of
a general enough Heisenberg invariant Kummer surface is not in
$\pi(I_{(d,g)})$, we have that $\pi(I_{(d,g)}) \neq U$ if $d$ is
odd. A surface whose parameter point is outside the countable union
$\bigcup_{k,g} \pi(I_{(2k+1,g)})$ has no curve of odd degree.
\end{proof}

Finite Abelian groups with a symplectic action on some K3 surface have
been classified by Nikulin \cite{Nik}. In that paper, Theorem 4.7,
it is shown that if $G$ is finite Abelian and acts symplectically on a
K3 surface $X$, then the induced action of $G$ on $H^2(X,\ZZ)$ is
independent of the surface, up to isomorphism of lattices. It follows
that the lattices $H^2(X,\ZZ)^G$ and $\Omega_G=(H^2(X,\ZZ)^G)^{\perp}$
are determined by $G$ up to isomorphism. In \cite{Sa} these lattices
are worked out in each case of the classification and in the case of
the Heisenberg group, $\Omega_H$ turns out to be well known. It is
isomorphic to $-\Lambda_{15}$, where $\Lambda_{15}$ is the so-called
laminated lattice of rank $15$, see \cite{C}.

\begin{remark}
The discriminant of $\Lambda_{15}$ is $2^9$ and $\Lambda_{15}$ is
positive definite, see \cite{C} for a detailed treatment. In one
basis, the bilinear form on $\Lambda_{15}$ is given by
{\small
\begin{displaymath}
\left(
\begin{array}{*{15}K}
4&-2&0&0&0&0&0&0&0&0&0&0&0&1&1\\
-2&4&-2&2&0&0&0&0&0&0&0&0&-1&0&0\\
0&-2&4&0&0&2&0&0&0&0&0&0&2&1&1\\
0&2&0&4&2&2&0&0&0&0&0&0&0&1&1\\
0&0&0&2&4&2&0&0&2&1&0&0&0&0&0\\
0&0&2&2&2&4&2&2&1&2&0&0&1&1&2\\
0&0&0&0&0&2&4&2&0&2&0&0&0&-1&1\\
0&0&0&0&0&2&2&4&0&2&0&0&1&0&2\\
0&0&0&0&2&1&0&0&4&2&0&0&0&0&0\\
0&0&0&0&1&2&2&2&2&4&2&2&1&1&2\\
0&0&0&0&0&0&0&0&0&2&4&2&2&1&1\\
0&0&0&0&0&0&0&0&0&2&2&4&1&2&2\\
0&-1&2&0&0&1&0&1&0&1&2&1&4&0&2\\
1&0&1&1&0&1&-1&0&0&1&1&2&0&4&2\\
1&0&1&1&0&2&1&2&0&2&1&2&2&2&4
\end{array}
\right).
\end{displaymath}
}
\end{remark}

\begin{theorem}
Let $X$ be a very generic Heisenberg invariant quartic and let $h \in
\pic{X}$ be the hyperplane class. Then
\begin{enumerate}
\item the sublattice $\ZZ h  \oplus \Omega_H \subset
\pic{X}$ has index $2$,
\item $\discr{\pic{X}}=-2^9$.
\end{enumerate}
\end{theorem}
\begin{proof}
We first show that the group $\pic{X}/(\ZZ h \oplus \Omega_H)$ is
cyclic. Since $\pic{X}^H=\ZZ h$ and $H^2(X,\ZZ)^H$ has rank 7, $\ZZ h
\oplus T_X \subseteq H^2(X,\ZZ)^H$ is a full rank sublattice. If $l
\in \pic{X}$ and $lh=0$, then $l$ kills $\ZZ h \oplus T_X$ and
therefore $l \in \Omega_H$. Since $\Omega_H \subset \pic{X}$, it
follows that $\Omega_H$ is the orthogonal complement of $\ZZ h$ in
$\pic{X}$. Observe that $\Omega_H$ is a primitive sublattice of
$\pic{X}$, that is if $nl \in \Omega_H$ for some non-zero $n \in \ZZ$
and some $l \in \pic{X}$, then $l \in \Omega_H$. In other words
$\pic{X}/\Omega_H$ is torsion free. Because $\pic{X}/\Omega_H$ has
rank 1, there is an element $l \in \pic{X}$ such that $\pic{X}=\ZZ l
\oplus \Omega_H$, as Abelian groups. The class of $l$ in $\pic{X}/(\ZZ
h \oplus \Omega_H)$ generates $\pic{X}/(\ZZ h \oplus
\Omega_H)$. Further, note that for $v \in (\ZZ h \oplus \Omega_H)$,
$v=nh+\omega$ with $\omega \in \Omega_H$ and $n \in \ZZ$, we have that
$vh=nh^2=4n$. If it were true that $\pic{X}=\ZZ h \oplus \Omega_H$
then the degree of any curve on $X$ would be a multiple of 4, but this
is not the case since $X$ contains a conic. Now let $D \in
\pic{X}$. It follows by \lemmaref{lemma:odd_degree} that $Dh$ is even,
say $Dh=2m$. Then $(2D-mh)h=0$ and hence $(2D-mh) \in
\Omega_H$. Consequently, $2D \in \ZZ h \oplus \Omega_H$ and thus $\ZZ
h \oplus \Omega_H$ has index 2 in $\pic{X}$.

Since $$\discr{{\ZZ h \oplus \Omega_H}}=4\cdot
\discr{\Omega_H}=-2^{11}$$ and $$\discr{\ZZ h \oplus \Omega_H}/
\discr{\pic{X}}=2^2,$$ we have that $\discr{\pic{X}}=-2^9$.
\end{proof}

\begin{corollary}
The Picard group of a very generic Heisenberg invariant quartic $X$ is
generated by the 320 conics.
\end{corollary}
\begin{proof}
By \propref{prop:configuration}, some set of 16 conics on $X$
correspond to the 16 conics of \remref{rem:16conics}. Let $P \subseteq
\pic{X}$ be the sublattice generated by these 16 conics on
$X$. Because $\discr{P}=\discr{\pic{X}}$, the index of $P$ in
$\pic{X}$ is equal to 1, that is $P=\pic{X}$.
\end{proof}

The proof of the corollary shows that (in some basis) the lattice
structure of the Picard group of a very generic Heisenberg invariant
quartic is given by the matrix $M$ in \remref{rem:16conics}.

\end{document}